\documentclass[graybox]{svmult}
\usepackage{amssymb}
\usepackage{amsfonts}
\usepackage{amscd}
\usepackage{PSFigure,psfrag}

\usepackage{mathptmx}       
\usepackage{helvet}         
\usepackage{courier}        
\usepackage{type1cm}        
\usepackage{makeidx}         
\usepackage{graphicx}        
\usepackage{multicol}        
\usepackage[bottom]{footmisc}

\makeindex             

\def\Vvu{\stackrel{\wh{\cV}^u}{\longleftrightarrow} \hspace{-2.8ex} \mbox{\f /}\;\;\;}
\def\Vu{\stackrel{\cV^u}{\longleftrightarrow} \hspace{-2.8ex} \mbox{\f /}\;\;\;}

\def\vv{\stackrel{v}{\longleftrightarrow} \hspace{-2.8ex} \mbox{\f /}\;\;\;}
\def\u{\stackrel{u}{\longleftrightarrow} \hspace{-2.8ex} \mbox{\scriptsize /}\;\;\;}
\def\uv{\stackrel{u + \varepsilon}{\longleftrightarrow} \hspace{-2.8ex} \mbox{\f /}\;\;\;}
\def\uu{\stackrel{u}{\longleftrightarrow}}
\def\uuu{\stackrel{u'}{\longleftrightarrow} \hspace{-2.8ex} \mbox{\f /}\;\;\;}

\def\lla{\longleftrightarrow}

\def\IP{{\mathbb P}}
\def\IR{{\mathbb R}}

\def\IZ{{\mathbb Z}}

\def\n{\noindent}

\def\dis{\displaystyle}

\def\ov{\overline}
\def\ve{\varepsilon}
\def\f{\footnotesize}

\def\r{\rightarrow}
\def\point{{\mbox{\large $.$}}}
\def\wh{\widehat}
\def\wt{\widetilde}

\def\cC{{\cal C}}
\def\cD{{\cal D}}

\def\cI{{\cal I}}

\def\cV{{\cal V}}

\begin{document}

\title*{On the $C^1$-property of the percolation function of random interlacements and a related variational problem}
\author{Alain-Sol Sznitman}
\titlerunning{On the $C^1$-property of the percolation function}
\institute{Alain-Sol Sznitman \at Department of Mathematics, ETH Z\"urich, R\"amistrasse 101, 8092 Z\"urich, Switzerland,\\
 \email{sznitman@math.ethz.ch}
}
%
%

\maketitle

\vspace{-2cm}
\begin{center}
{\it In memory of Vladas Sidoravicius}
\end{center}

\bigskip
\abstract{We consider random interlacements on $\IZ^d$, $d \ge 3$. We show that the percolation function that to each $u \ge 0$  attaches the probability that the origin does not belong to an infinite cluster of the vacant set at level $u$, is $C^1$ on an interval $[0,\wh{u})$, where $\wh{u}$ is positive and plausibly coincides with the critical level $u_*$ for the percolation of the vacant set. We apply this finding to a constrained minimization problem that conjecturally expresses the exponential rate of decay of the probability that a large box contains an excessive proportion $\nu$ of sites that do not belong to an infinite cluster of the vacant set. When $u$ is smaller than $\wh{u}$, we describe a regime of ``small excess'' for $\nu$ where all minimizers of the constrained minimization problem remain strictly below the natural threshold value $\sqrt{u}_* - \sqrt{u}$ for the variational problem.
}



\setcounter{section}{-1}
\section{Introduction}

In this work we consider random interlacements on $\IZ^d$, $d \ge 3$, and the percolation of the vacant set of random interlacements. We show that the percolation function $\theta_0$ that to each level $u \ge 0$ attaches the probability that the origin does not belong to an {\it infinite} cluster of $\cV^u$, the vacant set at level $u$ of the random interlacements, is $C^1$ on an interval $[0,\wh{u})$, where $\wh{u}$ is positive and plausibly coincides with the critical level $u_*$ for the percolation of $\cV^u$, although this equality is presently open. We apply this finding to a constrained minimization problem that for $0<u<u_*$ conjecturally expresses the exponential rate of decay of the probability that a large box contains an excessive proportion $\nu$ bigger than $\theta_0(u)$ of sites that do not belong to the infinite cluster of $\cV^u$. When $u > 0$ is smaller than $\wh{u}$ and $\nu$ close enough to $\theta_0(u)$, we show that all minimizers $\varphi$ of the constrained minimization problem are $C^{1,\alpha}$-functions on $\IR^d$, for all $0 < \alpha < 1$, and their supremum norm  lies strictly below $\sqrt{u}_* - \sqrt{u}$. In particular, the corresponding ``local level'' functions $(\sqrt{u} + \varphi)^2$ do not reach the critical value $u_*$.

We now discuss our results in more details. We consider random interlacements on $\IZ^d$, $d \ge 3$, and refer to \cite{CernTeix12} or \cite{DrewRathSapo14c} for background material. For $u \ge 0$, we let $\cI^u$ stand for the random interlacements at level $u$ and $\cV^u = \IZ^d \backslash \cI^u$ for the vacant set at level $u$. A key object of interest is the percolation function 
\begin{equation}\label{0.1}
\theta_0(u) = \IP [0 \u  \;\infty], \;\mbox{for $u \ge 0$},
\end{equation}
where $\{0 \u \; \infty\}$ is a shorthand for the event $\{0 \Vu \; \infty\}$ stating that $0$ does not belong to an infinite cluster of $\cV^u$. One knows from \cite{Szni10} and \cite{SidoSzni09} that there is a critical value $u_* \in (0, \infty)$ such that $\theta_0$ equals $1$ on $(u_{*,\infty})$ and is smaller than $1$ on $(0,u_*)$. And from Corollary 1.2 of \cite{Teix09a}, one knows that the non-decreasing left-continuous function $\theta_0$ is continuous except maybe at the critical value $u_*$.

\begin{figure}[h]
\begin{center}
\psfrag{0}{$0$}
\psfrag{u}{$u$}
\psfrag{u*}{$u_*$}
\psfrag{1}{$1$}
\psfrag{t0}{$\theta_0$}
\includegraphics[scale=.45]{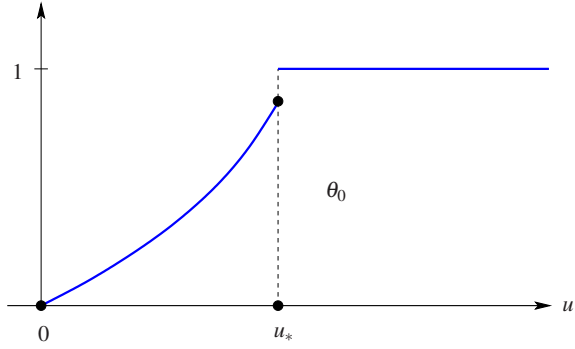}
\\
\caption{A heuristic sketch of the graph of $\theta_0$ (with a possible but not expected jump at $u_*$)}
\label{fig:1}       
\end{center}
\end{figure}

With an eye towards applications to a variational problem that we discuss below, see (\ref{0.9}), we are interested in proving that $\theta_0$ is $C^1$ on some (hopefully large) neighborhood of $0$. With this goal in mind, we introduce the following definition. Given $0 \le \alpha < \beta < u_*$, we say that NLF$ (\alpha, \beta)$, the no large finite cluster property on $[\alpha,\beta]$, holds when
\begin{equation}\label{0.2}
\begin{array}{l}
\mbox{there exists $L_0(\alpha,\beta) \ge 1$, $c_0 (\alpha,\beta) > 0$, $\gamma(\alpha, \beta) \in (0,1]$ such that}
\\
\mbox{for all $L \ge L_0$ and $u \in [\alpha,\beta]$, $\IP [0 \stackrel{u}{\longleftrightarrow} \partial B_L$, $0 \u \; \infty] \le e^{-c_0 L^\gamma}$},
\end{array}
\end{equation}
where $B_L = B(0,L)$ is the closed ball for the sup-norm with center $0$ and radius $L$, $\partial B_L$ its internal boundary (i.e. the subset of sites in $B_L$  that are neighbors of $\IZ^d \backslash B_L$), and the notation is otherwise similar to (\ref{0.1}). We then set
\begin{equation}\label{0.3}
\wh{u} = \sup\{u \in [0,u_*) \,; \; {\rm NLF}(0,u) \;\mbox{holds}\}.
\end{equation}
One knows from Corollary 1.2 of  \cite{DrewRathSapo14a} that $\wh{u}$ is positive:
\begin{equation}\label{0.4}
\wh{u} \in (0,u_*]\,.
\end{equation}

It is open, but plausible, that $\wh{u} = u_*$ (see also \cite{DumGosRodrSev} for related progress in the context of level-set percolation of the Gaussian free field). Our first main result is:
\begin{theorem}\label{theo0.1}
\begin{eqnarray}
&&\mbox{The function $\theta_0$ is $C^1$ on $[0,\wh{u})$ and}\label{0.5}
\\
&&\mbox{$\theta'_0$ is positive on $[0,\wh{u})$}.\label{0.6}
\end{eqnarray}
\end{theorem}

Incidentally, let us mention that in the case of Bernoulli percolation the function corresponding to $\theta_0$ is known to be $C^\infty$ in the supercritical regime, see Theorem 8.92 of \cite{Grim99}. However, questions pertaining to the sign of the second derivative (in particular the possible convexity of the corresponding function in the supercritical regime) are presently open. Needless to say that in our case the shape of the function $\theta_0$ is not known (and the sketch in Figure 1 conceivably misleading).

Our interest in Theorem \ref{theo0.1} comes in conjunction with an application to a variational problem that we now describe. We consider
\begin{equation}\label{0.7}
\begin{array}{l}
\mbox{$D$ the closure of a smooth bounded domain, or of an open}
\\
\mbox{sup-norm ball, of $\IR^d$ that contains $0$}.
\end{array}
\end{equation}
Given $u$ and $\nu$ such that
\begin{equation}\label{0.8}
\mbox{$0 < u < u_*$ and $\theta_0(u) \le \nu < 1$},
\end{equation}
we introduce the constrained minimization problem
\begin{equation}\label{0.9}
I^D_{u,\nu} = \inf \Big\{\mbox{\f $\dis\frac{1}{2d}$} \;\dis\int_{\IR^d} | \nabla \varphi|^2 dz; 
\varphi \ge 0, \varphi \in C^\infty_0 (\IR^d), \;\dis\int_D \hspace{-2.5ex}{-} \;
\theta_0\big((\sqrt{u} + \varphi)^2\big) \,dz > \nu\Big\},
\end{equation}
where $C^\infty_0(\IR^d)$ stands for the set of smooth compactly supported functions on $\IR^d$ and $\int_D \hspace{-2.5ex}{-} \; \dots dz$ for the normalized integral $\frac{1}{|D|} \int \dots dz$ with $|D|$ the Lebesgue measure of $D$ (see also below (\ref{0.10}) for the interpretation of $\varphi$).

The motivation for the variational problem (\ref{0.9}) lies in the fact that it conjecturally describes the large deviation cost of having a fraction at least $\nu$ of sites in the large discrete blow-up $D_N = (ND) \cap \IZ^d$ of $D$ that are not in the infinite cluster $\cC^u_\infty$ of $\cV^u$. One knows by the arguments of Remark 6.6~2) of \cite{Szni19} that
\begin{equation}\label{0.10}
\liminf\limits_N \; \mbox{\f $\dis\frac{1}{N^{d-2}}$}\; \log \IP[ | \,D_N \backslash \cC^u_\infty| \ge \nu \, |D_N|] \ge - I^D_{u,\nu} \;\; \mbox{for $u,\nu$ as in (\ref{0.8})}
\end{equation}
(with $|A|$ standing for the number of sites in $A$ for $A$ subset of $\IZ^d$).

\bigskip
It is presently open whether the lim inf can be replaced by a limit and the inequality by an equality in (\ref{0.10}), i.e. if there is a matching asymptotic upper bound. If such is the case, there is a direct interest in the introduction of a notion of minimizers for (\ref{0.9}). Indeed, $(\sqrt{u} + \varphi)^2 (\frac{\cdot}{N})$ can be interpreted as the slowly varying local levels of the tilted interlacements that enter the derivation of the lower bound (\ref{0.10}) (see Section 4 and Remark 6.6~2) of \cite{Szni19}). In this perspective, it is a relevant question whether minimizers $\varphi$ reach the value $\sqrt{u}_* - \sqrt{u}$. The regions where they reach the value $\sqrt{u_*} - \sqrt{u}$ could potentially reflect the presence of droplets secluded from the infinite cluster $\cC^u_\infty$ and taking a share of the burden of creating an excess fraction $\nu$ of sites of $D_N$ that are not in $\cC^u_\infty$ (see also the discussion at the end of Section 2).

\bigskip
The desired notion of minimizers for (\ref{0.9}) comes in Theorem \ref{theo0.2} below. For this purpose we introduce the right-continuous modification $\ov{\theta}_0$ of $\theta_0$:
\begin{equation}\label{0.11}
\ov{\theta}_0(u) = \left\{ \begin{array}{ll}
\theta_0(u), &\mbox{when $0 \le u < u_*$},
\\[1ex]
1, & \mbox{when $u \ge u_*$}.
\end{array}\right. 
\end{equation}
Clearly, $\ov{\theta}_0 \ge \theta_0$ and it is plausible, but presently open, that $\ov{\theta}_0 = \theta_0$. We recall that $D^1(\IR^d)$ stands for the space of locally integrable functions $f$ on $\IR^d$ with finite Dirichlet energy that decay at infinity, i.e.~such that $\{|f| > a\}$ has finite Lebesgue measure for all $a > 0$, see Chapter 8 of \cite{LiebLoss01}, and define for $D,u,\nu$ as in (\ref{0.7}), (\ref{0.8})
\begin{equation}\label{0.12}
\ov{J}\,^{\!D}_{u,\nu} = \inf\Big\{ \mbox{\f $\dis\frac{1}{2d}$} \; \dis\int_{\IR^d} | \nabla \varphi|^2\, dz; \varphi \ge 0, \varphi \in D^1 (\IR^d), \; \dis\int_D \hspace{-2.5ex}{-} \;
\ov{\theta}_0 \big((\sqrt{u} + \varphi)^2\big)\,dz \ge \nu\Big\}.
\end{equation}
Since $\ov{\theta}_0 \ge \theta_0$ and $D^1(\IR^d) \supseteq C^\infty_0(\IR^d)$, we clearly have $\ov{J}\,^{\!D}_{u,\nu} \le I^D_{u,\nu}$. But in fact:
\begin{theorem}\label{theo0.2}
For $D,u, \nu$ as in (\ref{0.7}), (\ref{0.8}), one has
\begin{equation}\label{0.13}
\ov{J}\,^{\!D}_{u,\nu} = I^D_{u,\nu}.
\end{equation}
In addition, the infimum in (\ref{0.12}) is attained:
\begin{equation}\label{0.14}
\begin{array}{l}
\ov{J}\,^{\!D}_{u,\nu} =  \min \Big\{ \mbox{\f $\dis\frac{1}{2d}$} \; \dis\int_{\IR^d} | \nabla \varphi|^2 \, dz; \varphi \ge 0, 
\\[1ex]
\hspace{4cm} \varphi \in D^1(\IR^d), \;\dis\int_D \hspace{-2.5ex}{-} \;
\ov{\theta}_0 \big((\sqrt{u} + \varphi)^2\big)\,dz \ge \nu\Big\}.
\end{array}
\end{equation}
and any minimizer $\varphi$ in (\ref{0.14}) satisfies
\begin{equation}\label{0.15}
\begin{array}{l}
\mbox{$0 \le \varphi \le \sqrt{u}_* - \sqrt{u}$ a.e.,}
\\[1ex]
\mbox{$\varphi$ is harmonic outside $D$, and ${\rm ess}\sup\limits_{\hspace{-3ex}z \in \IR^d} |z|^{d-2} \varphi(z) < \infty$}.
\end{array}
\end{equation}
\end{theorem}

Thus, Theorem \ref{theo0.2} provides a notion of minimizers for (\ref{0.9}), the variational problem of interest. Its proof is given in Section 2. Additional properties of (\ref{0.14}) and the corresponding minimizers can be found in Remark \ref{rem2.1}. We refer to Chapter 11 \S3 of \cite{AmbrMalc07} for other instances of non-smooth variational problems.

\bigskip
In Section 3 we bring into play the $C^1$-property of $\theta_0$ and show

\pagebreak
\begin{theorem}\label{theo0.3}
If $u_0 \in (0, u_*)$ is such that
\begin{equation}\label{0.16}
\mbox{$\theta_0$ is $C^1$ on a neighborhood of $[0,u_0]$},
\end{equation}
then for any $u \in (0, u_0)$ there are $c_1(u,u_0, D) <  \theta_0(u_*) - \theta_0(u)$ and $c_2(u,u_0) > 0$ such that
\begin{equation}\label{0.17}
\begin{array}{l}
\mbox{for $\nu \in [\theta_0(u), \theta_0(u) + c_1]$, any minimizer $\varphi$ in (\ref{0.14}) is $C^{1,\alpha}$ for all}
\\
\mbox{$0 < \alpha < 1$, and $0 \le \varphi \le \big\{c_2 \big(\nu - \theta_0(u)\big) \big\} \wedge (\sqrt{u}_0 - \sqrt{u}) \;(< \sqrt{u}_* - \sqrt{u})$. }
\\
\mbox{Here $C^{1,\alpha}$ stands for the $C^1$-functions with $\alpha$-H\"older continuous partial}
\\
\mbox{derivatives.}
\end{array}                  
\end{equation}
\end{theorem}

In view of Theorem \ref{theo0.1} the above Theorem \ref{theo0.3} applies to any $u_0 < \wh{u}$ (with $\wh{u}$ as in (\ref{0.3})). It describes a regime of ``small excess'' for $\nu$ where minimizers do not reach the threshold value $\sqrt{u}_* - \sqrt{u}$. In the proof of Theorem \ref{theo0.3} we use the $C^1$-property to write an Euler-Lagrange equation for the minimizers, see (\ref{3.19}), and derive a bound in terms of $\nu - \theta_0(u)$ of the corresponding Lagrange multipliers, see (\ref{3.20}). It is an interesting open problem whether a regime of ``large excess'' for $\nu$ can be singled out where some (or all) minimizers of  (\ref{0.14}) reach the threshold value $\sqrt{u}_* - \sqrt{u}$ on a set of positive Lebesgue measure. We refer to Remark \ref{rem3.4} for some simple minded observations related to this issue.

Finally, let us state our convention about constants. Throughout we denote by $c,c',\wt{c}$ positive constants changing from place to place that simply depend on the dimension $d$. Numbered constants $c_0,c_1,c_2,\dots$ refer to the value corresponding to their first appearance in the text. Dependence on additional parameters appears in the notation.

\section{The $C^1$-property of $\theta_0$}

The main object of this section is to prove Theorem \ref{theo0.1} stated in the Introduction. Theorem \ref{theo0.1} is the direct consequence of the following Lemma \ref{lem1.1} and Proposition \ref{prop1.2}. We let $g(\cdot,\cdot)$ stand for the Green function of the simple random walk on $\IZ^d$.
\begin{lemma}\label{lem1.1}
For $0 \le u < u_*$, one has
\begin{equation}\label{1.1}
\liminf\limits_{\ve \downarrow 0} \; \mbox{\f $\dis\frac{1}{\ve}$} \;\big(\theta_0 (u  + \ve) - \theta_0(u)\big) \ge \big(1- \theta_0(u)\big) \; \mbox{\f $\dis\frac{1}{g(0,0)}$}.
\end{equation}
\end{lemma}

\begin{proposition}\label{prop1.2}
For any $0 \le \alpha < \beta < u_*$ such that NLF$(\alpha,\beta)$ holds (see (\ref{0.2})), 
\begin{equation}\label{1.2}
\mbox{$\theta_0$ is $C^1$ on $\Big[\alpha, \mbox{\f $\dis\frac{\alpha + \beta}{2}$}\Big]$}.
\end{equation}
\end{proposition}

As we now explain, Theorem \ref{theo0.1} follows immediately. By Proposition \ref{prop1.2} and a covering argument, one see that $\theta_0$ is $C^1$ on $[0,\wh{u})$. Then, by Lemma \ref{lem1.1}, one finds that $\theta'_0 > 0$ on $[0,\wh{u})$, and Theorem \ref{theo0.1} follows.

There remains to prove Lemma \ref{lem1.1} and Proposition \ref{prop1.2}.

\bigskip\n
{\it Proof of Lemma \ref{lem1.1}}: Consider $u \ge 0$ and $\ve > 0$ such that $u + \ve < u_*$. Then, denoting by $\cI^{u,u + \ve}$ the collection of sites of $\IZ^d$ that are visited by trajectories of the interlacement with level lying in $(u,u+ \ve]$, we have
\begin{equation}\label{1.3}
\begin{array}{lll}
\theta_0(u + \ve) - \theta_ 0(u) & = &\hspace{-3ex}  \IP[ 0 \uu \; \infty, \; 0 \uv \; \infty]
\\
&\ge&\hspace{-3ex} \IP [ 0 \uu \;\infty, \;0 \in \cI^{u,u + \ve}]
\\
&\hspace{-4ex}\stackrel{\rm independence}{=}& \big(1 - \theta_0(u)\big) \, \IP[0 \in \cI^{u,u + \ve}] 
\\
&\!=&\hspace{-3ex} \big(1 - \theta_0(u)\big) (1 - e^{-\ve / g(0,0)}).
\end{array}
\end{equation}
Dividing by $\ve$ both members of (\ref{1.3}) and letting $\ve$ tend to $0$ yields (\ref{1.1}). This proves Lemma \ref{lem1.1}. \hfill $\square$

We now turn to the proof of Proposition \ref{prop1.2}. An important tool is Lemma \ref{lem1.3} below. We will use Lemma  \ref{lem1.3} to gain control over the difference quotients of $\theta_0$, as expressed in  (\ref{1.10}) or  (\ref{1.20}) below. The claimed $C^1$-property of $\theta_0$ on $[\alpha, \frac{\alpha + \beta}{2}]$ will then quickly follow, see below  (\ref{1.20}). To prove  (\ref{1.10}) with Lemma \ref{lem1.3}, we define an increasing sequence of levels $u_i$, $1 \le i  \le i_\eta$ so that $u_1 = u'$ (in Proposition  \ref{prop1.2}) and $u_i - u$ doubles as $i$ increases by one unit, until it reaches $\eta$ (of (\ref{1.10})), and in essence apply Lemma \ref{lem1.3} repeatedly to compare the successive difference quotients of $\theta_0$ between $u$ and $u_i$, see (\ref{1.15}) till  (\ref{1.19}).

{\it Proof of Proposition \ref{prop1.2}}: We consider $0 \le \alpha, \beta < u_*$ such that NLF$(\alpha, \beta)$ holds (see (0.2)), and set
\begin{equation}\label{1.4}
c_3 \,(\alpha, \beta) = 2 / c_0 \,.
\end{equation}
As mentioned above, an important tool in the proof of Proposition \ref{prop1.2} is provided by

\begin{lemma}\label{lem1.3}
Consider $u < u' \le u''$ in $[\alpha, \frac{\alpha + \beta}{2}]$ such that
\begin{equation}\label{1.5}
u'' - u \le e^{- \frac{1}{c_3} \,L_0^\gamma} \; (\le 1),
\end{equation}
and set
\begin{equation}\label{1.6}
\Delta' = \mbox{\f $\dis\frac{1}{u' - u}$} \;\big(\theta_0(u') - \theta_0(u)\big) \;\;\mbox{and} \; \; \Delta'' = \mbox{\f $\dis\frac{1}{u'' - u}$} \; \big(\theta_0(u'') - \theta_0(u)\big),
\end{equation}
as well as $L' \ge L'' \ge L_0$ (with $L_0$ as in (\ref{0.2})) via
\begin{equation}\label{1.7}
L' = \Big(c_3 \; \log \; \mbox{\f $\dis\frac{1}{u' - u}$}\Big)^{1/\gamma} \;\; \mbox{and} \;\; L'' =  \Big(c_3 \; \log \; \mbox{\f $\dis\frac{1}{u'' - u}$}\Big)^{1/\gamma} .
\end{equation}
Then, with ${\rm cap}(\cdot)$ denoting the simple random walk capacity, one has 
\begin{equation}\label{1.8}
|\Delta ' - e^{(u'' - u') \,{\rm cap}(B_{L'})}  \,\Delta'' | \le 3(u'' - u) \big(1 + {\rm cap}(B_{L'})^2\big) \,e^{(u'' - u') \,{\rm cap}(B_{L'})}.
\end{equation}
\end{lemma}

\pagebreak
Let us first admit Lemma \ref{lem1.3} and conclude the proof of Proposition \ref{prop1.2} (i.e. that $\theta_0$ is $C^1$ on $[\alpha, \frac{\alpha + \beta}{2}]$). We introduce
\begin{equation}\label{1.9}
\eta_0 = \mbox{\f $\dis\frac{1}{4}$} \;\Big(\mbox{\f $\dis\frac{\beta - \alpha}{2}$} \wedge e^{-\frac{1}{c_3} \,L_0^\gamma}\Big) \;\;\Big( \le \mbox{\f $\dis\frac{1}{4}$} \Big).
\end{equation}

\bigskip\noindent
We will use Lemma \ref{lem1.3} to show that
\begin{equation}\label{1.10}
\begin{array}{l}
\mbox{when $0 < \eta \le \eta_0$, then for all $u < u'$ in $\Big[\alpha, \mbox{\f $\dis\frac{\alpha, \beta}{2}$}\Big]$ with $u' \le u + \eta$, one has}
\\
\Big| \mbox{\f $\dis\frac{1}{u' - u}$} \;\big(\theta_0(u') - \theta_0(u)\big) - \mbox{\f $\dis\frac{1}{\eta}$} \;\big(\theta_0(u + \eta) - \theta_0(u)\big)\Big| \le c(\alpha,\beta) \,\sqrt{\eta}.
\end{array}
\end{equation}

\medskip\noindent
Once (\ref{1.10}) is established, Proposition \ref{prop1.2} will quickly follow (see below (\ref{1.20})). For the time being we will prove (\ref{1.10}). To this end we set
\begin{equation}\label{1.11}
\mbox{$u_i = 2^{i - 1} (u' - u) + u$, for $1 \le i \le i_\eta$, where $i_\eta = \max\{ i \ge 1, u_i \le u + \eta\}$}
\end{equation}
(note that $u_1 = u')$, as well as
\begin{equation}\label{1.12}
\begin{array}{l}
\mbox{$\Delta_i = \mbox{\f $\dis\frac{1}{u_i - u} $} \;\big(\theta_0(u_i) - \theta_0(u)\big) \;\; \mbox{and} \;\;L_i = \Big(c_3 \,\log \;\mbox{\f $\dis\frac{1}{u_i - u}$}\Big)^{\frac{1}{\gamma}} \;(\stackrel{(\ref{1.9})}{\ge} L_0)$,}
\\[2ex]
\qquad \mbox{for $1 \le i \le i_\eta$.}
\end{array}
\end{equation}
We also define
\begin{equation}\label{1.13}
\delta_i = (u_i - u) \,{\rm cap}(B_{L_i}) \;\;\mbox{and} \;\; \wt{\delta}_i = 6(u_i - u) + 6 \delta_i \,{\rm cap}(B_{L_i}),\;
\mbox{for} \;\; 1 \le i \le i_\eta.
\end{equation}
We will apply (\ref{1.8}) of Lemma \ref{lem1.3} to $u' = u_i$, $u'' = u_{i+1}$, when $1 \le i < i_\eta$, and to $u' = u_{i_\eta}$, $u'' = u + \eta$. We note that for $1 \le i \le i_\eta$, we have $\delta_i \le c(\alpha, \beta) \,\sqrt{u_i - u}$ and $\wt{\delta}_i \le c(\alpha, \beta) \,\sqrt{u_i - u}$ so that
\begin{equation}\label{1.14}
\mbox{for $1 \le j \le i_\eta$, $\dis\sum\limits_{1 \le i \le j} \delta_i \le c(\alpha, \beta)\, \sqrt{u_j - u}$ and $\dis\sum\limits_{1 \le i \le j} \wt{\delta}_i \le c(\alpha, \beta) \,\sqrt{u_j - u}$}.
\end{equation}
The application of (\ref{1.8}) to $u'' = u_{i+1}$, $u' = u_i$, for $1 \le i < i_\eta$ and the observation that $u_{i+1} - u_i = u_i - u$ yield the inequality
\begin{equation}\label{1.15}
\begin{array}{l}
\mbox{$| \Delta_i - e^{\delta_i} \Delta_{i + 1}| \le c \,\wt{\delta}_i\,e^{\delta_i}$, for $1 \le i < i_\eta$, so that}
\\
\big| e^{\sum_{\ell < i} \,\delta_\ell} \Delta_i - e^{\sum_{\ell < i + 1} \delta_\ell} \Delta_{i+1}\big| \le c \,e^{\sum_{\ell < i + 1} \delta_\ell}\, \wt{\delta}_i, \;\mbox{for $1 \le i < i_\eta$}.
\end{array}
\end{equation}
Hence, adding these inequalities, we find that
\begin{equation}\label{1.16}
\big| \Delta_1 - e^{\sum_{\ell < i_\eta} \delta_\ell} \Delta_{i_\eta} \big| \le c \, \dis\sum\limits_{1 \le i < i_\eta}  e^{\sum_{\ell < i +1} \delta_\ell} \;\wt{\delta}_i \stackrel{(\ref{1.14}), \eta \le \frac{1}{4}}{\le} c(\alpha, \eta) \, \sqrt{\eta}\,.
\end{equation}
Then, the application of (\ref{1.8}) to $u'' = u + \eta$ and $u' = u_{i_\eta}$, noting that $u + \eta - u_{i_\eta} \le u_{i_\eta} - u$, yields
\begin{equation}\label{1.17}
\Big|\Delta_{i_\eta}  - e^{(u + \eta - u_{i_\eta}) \,{\rm cap}(B_{L_{i_\eta}})} \;\mbox{\f $\dis\frac{1}{\eta}$} \;\big(\theta_0(u + \eta) - \theta_0(u)\big)\Big|  \le \wt{\delta}_{i_\eta} \;e^{\delta_{i_\eta}}
 \le c(\alpha, \beta) \, \sqrt{\eta}.
\end{equation}
Multiplying both members of (\ref{1.17}) by $e^{\sum_{\ell < i_\eta} \delta_\ell}$ and using (\ref{1.16}) and (\ref{1.14}) as well, we thus find
\begin{equation}\label{1.18}
\begin{array}{l}
\Big| \mbox{\f $\dis\frac{1}{u' - u}$} \;\big(\theta_0(u') - \theta_0 (u)\big) - e^{\sum_{\ell < i_\eta} \delta_\ell + (u + \eta - u_{i_\eta}) \,{\rm cap}(B_{L_{i_\eta}})} \mbox{\f $\dis\frac{1}{\eta}$} \;\big(\theta_0(u + \eta) - \theta_0(u)\big)\Big| 
\\[2ex]
\le c(\alpha,\beta) \, \sqrt{\eta}
\end{array}
\end{equation}
and the term inside the exponential is at most $c(\alpha,\beta) \,\sqrt{\eta}$.

\bigskip
Applying (\ref{1.18}) with the choice $\eta = \eta_0$, see (\ref{1.9}), one obtains that
\begin{equation}\label{1.19}
\sup\limits_{\alpha \le u < u' \le \frac{\alpha + \beta}{2}\,, u' \le u + \eta_0} \;\mbox{\f $\dis\frac{1}{u' - u}$} \; \big(\theta_0(u') - \theta_0(u)\big) \le c(\alpha,\beta).
\end{equation}
Coming back to (\ref{1.18}), with the help of the observation below (\ref{1.18}) and the inequality $e^a - 1 \le c'(\alpha,\beta) a$ for $0 \le a \le c(\alpha,\beta)$, one obtains the claim (\ref{1.10}).

We will now see how the $C^1$-property of $\theta_0$ on $[\alpha, \frac{\alpha + \beta}{2}]$ (i.e. Proposition \ref{prop1.2}) follows. We note that for $v,w \in [\alpha, \frac{\alpha + \beta}{2}]$ with $0 < |v - w| \le \eta(\le \eta_0)$, the claim (\ref{1.10}) applied to $u = v \wedge w$ and $u' = v \vee w$ yields that
\begin{equation}\label{1.20}
\Big| \mbox{\f $\dis\frac{1}{w-v}$} \;\big(\theta_0(w) - \theta_0(v)\big) - \mbox{\f $\dis\frac{1}{\eta}$} \;\big(\theta_0 (v \wedge w + \eta) - \theta_0(v \wedge w)\big) \Big| \le c(\alpha,\beta) \,\sqrt{\eta}.
\end{equation}

Letting $\Gamma(\cdot)$ stand for the modulus of continuity of $\theta_0$ on the interval $[\alpha, \frac{\alpha + \beta}{2}] \subseteq [0,u_*)$, we find that for $v,w \in [\alpha, \frac{\alpha + \beta}{2}]$ with $0 < |v - w| \le \eta$ $(\le \eta_0)$, one has
\begin{equation}\label{1.21}
\begin{array}{l}
\Big| \mbox{\f $\dis\frac{1}{w-v}$} \;\big(\theta_0(w) - \theta_0(v)\big) - \mbox{\f $\dis\frac{1}{\eta}$} \;\big(\theta_0 (v  + \eta) - \theta_0(v)\big) \Big| \le 
\\
c(\alpha,\beta) \,\sqrt{\eta} + \mbox{\f $\dis\frac{2}{\eta}$} \;\Gamma(|w - v|).
\end{array}
\end{equation}

The above inequality implies that for any $v \in [\alpha, \frac{\alpha + \beta}{2}]$, when $w \in [\alpha, \frac{\alpha + \beta}{2}]$ tends to $v$, the difference quotients $\frac{1}{w-v} \,(\theta_0(w) - \delta_0(v))$ are Cauchy. Thus, letting $w$ tend to $v$, we find that
\begin{equation}\label{1.22}
\begin{array}{l}
\mbox{$\theta_0$ is differentiable on $[\alpha, \frac{\alpha + \beta}{2}]$, and for $0 < \eta \le \eta_0$ and $v \in [\alpha, \frac{\alpha + \beta}{2}]$},
\\
\Big| \theta'_0(v) - \mbox{\f $\dis\frac{1}{\eta}$} \;\big(\theta_0(v + \eta) - \theta_0(v)\big) \Big| \le c(\alpha,\beta) \,\sqrt{\eta}.
\end{array}
\end{equation}
As a result we see that $\theta'_0$ is the uniform limit on $[\alpha, \frac{\alpha + \beta}{2}]$ of continuous functions, and as such $\theta'_0$ is continuous. This is the claimed $C^1$-property of Proposition \ref{prop1.2}. The last missing ingredient is the

\medskip
{\it Proof of Lemma \ref{lem1.3}}: We introduce the notation for $v \ge 0$ and $L \ge 1$
\begin{equation}\label{1.23}
\theta_{0,L}(v) = \IP[0 \vv \partial B_L],
\end{equation}

\medskip\noindent
and the approximations of $\Delta'$ and $\Delta''$ in (\ref{1.6})
\begin{equation}\label{1.24}
\wt{\Delta}' = \mbox{\f $\dis\frac{1}{u' - u}$} \;\big(\theta_{0,L'}(u') - \theta_{0,L'}(u)\big), \;\wt{\Delta}'' =  \mbox{\f $\dis\frac{1}{u'' - u}$} \;\big(\theta_{0,L''}(u'') - \theta_{0,L''}(u)\big),
\end{equation}
where we recall that $L' \ge L''$ $(\ge L_0)$ are defined in (\ref{1.7}). Note that
\begin{equation}
\begin{array}{lll}
\Delta ' & = &\mbox{\f $\dis\frac{1}{u' - u}$} \;\big(\IP[0 \uuu \; \infty]  - \IP[0 \u \infty]\big) =  \mbox{\f $\dis\frac{1}{u' - u}$} \; \IP [0 \stackrel{u}{\longleftrightarrow} \infty,  0 \uuu \; \infty],\label{1.25}
\\[2ex]
\wt{\Delta} ' & = &\mbox{\f $\dis\frac{1}{u' - u}$} \;\IP[0 \uu \; \partial B_{L'},  0 \uuu \;\partial B_{L'}], \label{1.26}
\end{array}
\end{equation}
and as we now explain

\begin{equation}\label{1.27}
\Delta' - \wt{\Delta} ' = \mbox{\f $\dis\frac{1}{u' - u}$} \big(\IP[0 \stackrel{u'}{\lla} \partial B_{L'}, \; 0 \uuu \; \infty] - \IP[0 \uu \; \partial B_{L'}, 0 \u \; \infty]\big) .
\end{equation}
Indeed, by (\ref{1.25}), (\ref{1.26}), one has
\begin{equation}\label{1.28}
\begin{array}{lll}
\Delta' - \wt{\Delta} ' &= &\; \mbox{\f $\dis\frac{1}{u' - u}$} \big(\IP[0 \uu\; \infty,  0 \uuu \infty] - \IP[0\uu \partial B_{L'}, \; 0 \uuu \; \partial B_{L'}]\big) =
\\
&&\; \mbox{\f $\dis\frac{1}{u' - u}$} \;\big(\IP [ 0 \uu  \partial B_{L'}, 0 \uuu \;  \infty] - \IP[0 \uu \partial B_{L'}, \; 0 \u  \; \infty, \;0 \uuu\; \infty] 
\\
&& -\; \IP [0 \uu \partial B_{L'}, 0 \uuu  \; \partial B_{L'}]\big) =  \mbox{\f $\dis\frac{1}{u' - u}$} \;\big(\IP [0 \uu   \partial B_{L'}, 0 \uuu \;  \partial B_{L'}] 
\\
&&+ \;\IP[0 \uu   \partial B_{L'},  0 \stackrel{u'}{\lla} \partial B_{L'},  0\uuu \;\infty]
\\
&& - \;\IP [0 \uu  \partial B_{L'}, 0 \u \;  \infty] - \IP [0 \stackrel{u}{\lla} \partial B_{L'}, 0 \uuu \;  \partial B_{L'}]) =
\\
&&\; \mbox{\f $\dis\frac{1}{u' - u}$} \;\big(\IP [ 0\stackrel{u'}{\longleftrightarrow}  \partial B_{L'},  0 \uuu \;\infty]- \IP[0 \stackrel{u}{\longleftrightarrow}   \partial B_{L'},  0 \u \;  \infty]\big),
\end{array} 
\end{equation}

\medskip\noindent
whence (\ref{1.27}). Clearly, one also has similar identities as in (\ref{1.25}) - (\ref{1.27}) for $\Delta ''$ and $\wt{\Delta} ''$.

We now proceed with the proof of (\ref{1.8}). By (\ref{1.27}), we have
\begin{equation}\label{1.29}
\begin{array}{lll}
| \Delta' - \wt{\Delta} ' | & \le& \mbox{\f $\dis\frac{1}{u'-u}$} \;\max \big\{\IP [0 \stackrel{u'}{\lla} \partial B_{L'},  0 \uuu  \; \infty], \;\IP[0 \stackrel{u}{\lla} \partial B_{L'},   0 \u \;  \infty]\big\}
\\[2ex]
&\! \stackrel{(\ref{0.2})}{\le}& \mbox{\f $\dis\frac{1}{u'-u}$} \;e^{-c_0 \,L'\,^{\gamma} }\stackrel{(\ref{1.7}), (\ref{1.4})}{=} u' - u,
\end{array} 
\end{equation}
and likewise we have
\begin{equation}\label{1.30}
|\Delta '' - \wt{\Delta} ''| \le u'' - u.
\end{equation}

\medskip\noindent
We will now compare $\wt{\Delta}'$ and $\wt{\Delta} ''$. We first recall that when $Z$ is a Poisson-distributed random variable with parameter $\lambda > 0$, then one has
\begin{equation}\label{1.31}
P[Z \ge 2] = 1 - e^{-\lambda} - \lambda \, e^{-\lambda} = \dis\int^\lambda_0 s \, e^{-s} \,ds \le \mbox{\f $\dis\frac{\lambda^2}{2}$}\;.
\end{equation}
If $N_{u,u'} (B_{L'})$ stands for the number of trajectories in the interlacements with labels in $(u,u']$ that reach $B_{L'}$ (this is a Poisson$((u' - u) \,{\rm cap}(B_{L'}))$-distributed random variable), we find by (\ref{1.26}) that
\begin{equation}\label{1.32}
\begin{array}{l}
\wt{\Delta} ' = \mbox{\f $\dis\frac{1}{u' - u}$}\;\big(\IP [0 \uu  \partial B_{L'}, 0 \uuu \;  \partial B_{L'}, \, N_{u,u'}(B_{L'}) = 1] \; + 
\\
\IP[0 \uu   \partial B_{L'},   0 \uuu  \; \partial B_{L'}, \,N_{u,u'} (B_{L'}) \ge 2 ]\big).
\end{array}
\end{equation}
If we consider an independent random walk $X_\point$ with initial distribution $\ov{e}_{B_{L'}}$, where $\ov{e}_{B_{L'}}$ stands for the normalized equilibrium measure of $B_{L'}$, and write \hbox{$\wh{\cV}^u = \cV^u \backslash ({\rm range} \; X)$,} we find from (\ref{1.32}), (\ref{1.31}) that
\begin{equation}\label{1.33}
\begin{array}{l}
\big|\wt{\Delta} ' - {\rm cap}(B_{L'}) \,e^{-(u' - u) \,{\rm cap} (B_{L'})}  P_{\ov{e}_{B_{L'}}} \otimes \IP [0 \uu   \partial B_{L'}, \; 0 \Vvu  \; \partial B_{L'}] \big| \le 
\\
\mbox{\f $\dis\frac{1}{2}$} \;(u' - u) \,{\rm cap}(B_{L'})^2
\end{array}
\end{equation}
(this formula is close in spirit to Theorem 1 of \cite{DebePopo15}). Then, we note that 
\begin{equation}\label{1.34}
\begin{array}{l}
{\rm cap} (B_{L'}) \, e^{-(u' -u) \,{\rm cap}(B_{L'})} P_{\ov{e}_{B_{L'}}} \otimes \IP [0 \uu \partial B_{L'},\; 0 \Vvu \;\partial B_{L'}] =
\\
\mbox{\f $\dis\frac{1}{u''- u}$} \; e^{(u'' - u') \,{\rm cap}(B_{L'})} \IP[N_{u,u''}(B_{L'}) = 1] \, P_{\ov{e}_{B_{L'}}} \otimes \IP [0 \stackrel{u}{\lla} \partial B_{L'}, 0 \Vvu\; \partial B_{L'}] = 
\\
\mbox{\f $\dis\frac{1}{u''- u}$} \; e^{(u'' - u') \,{\rm cap}(B_{L'})} \big(\IP [0 \uu \partial B_{L'}, 0 \stackrel{u''}{{\longleftrightarrow} \hspace{-2.3ex} \mbox{\f /}\;\;\;} \partial B_{L'}] 
\\
- \;\IP[0 \uu \partial B_{L'}, 0  \stackrel{u''}{{\longleftrightarrow} \hspace{-2.3ex} \mbox{\f /}\;\;\;} \partial B_{L'}, \, N_{u,u''}(B_{L'}) \ge 2]\big).
\end{array}
\end{equation}
Inserting this identity into (\ref{1.33}) and using (\ref{1.31}) once again, we find that
\begin{equation}\label{1.35}
\begin{array}{l}
\Big| \wt{\Delta} ' - \mbox{\f $\dis\frac{1}{u''- u}$} \; e^{(u'' - u') \,{\rm cap}(B_{L'})} \IP[0 \uu  \partial B_{L'}, 0 \stackrel{u''}{{\longleftrightarrow} \hspace{-2.3ex} \mbox{\f /}\;\;\;} \partial B_{L'}]\Big| \le
\\
\mbox{\f $\dis\frac{1}{2}$} \;(u' - u) \,{\rm cap}(B_{L'})^2 + \mbox{\f $\dis\frac{1}{2}$}  \;(u'' - u) \, {\rm cap}(B_{L'})^2 \; e^{(u'' - u') \,{\rm cap}(B_{L'})} \le 
\\[2ex]
(u'' - u) \, {\rm cap}(B_{L'})^2 \,  e^{(u'' - u') \,{\rm cap}(B_{L'})}.
\end{array}
\end{equation}
Note that $L'' \le L'$ and a similar calculation as (\ref{1.28}) yields the identity
\begin{equation}\label{1.36}
\begin{array}{l}
\mbox{\f $\dis\frac{1}{u''- u}$} \;\IP[0 \uu \partial B_{L'} , 0  \stackrel{u''}{{\longleftrightarrow} \hspace{-2.3ex} \mbox{\f /}\;\;\;} \partial B_{L'}] - \wt{\Delta}'' =
\\
\mbox{\f $\dis\frac{1}{u''- u}$} \;(\IP[0 \stackrel{u''}{\lla} \partial B_{L''}, 0 \stackrel{u''}{{\longleftrightarrow} \hspace{-2.3ex} \mbox{\f /}\;\;\;} \partial B_{L'}] - \IP [0 \uu \partial B_{L''}, \; 0 \u \;\partial B_{L'}]\big)
\end{array}
\end{equation}
($u''$ plays the role of $u'$, $L''$ the role of $L'$, and $L'$ the role of $\infty$ in (\ref{1.27})). The application of (\ref{0.2}) with $L''$ as in (\ref{1.7}) now yields
\begin{equation}\label{1.37}
\Big| \mbox{\f $\dis\frac{1}{u'' - u}$} \; \IP[0 \uu \partial B_{L'}, \; 0 \stackrel{u''}{{\longleftrightarrow} \hspace{-2.3ex} \mbox{\f /}\;\;\;} \partial B_{L'}] - \wt{\Delta} ''\Big| \le u'' - u.
\end{equation}
%
Coming back to (\ref{1.35}), we find that
\begin{equation}\label{1.38}
| \wt{\Delta}' - e^{(u'' - u') \,{\rm cap}(B_{L'})} \;\wt{\Delta} '' | \le (u'' - u) \big(1 + {\rm cap} (B_{L'})^2\big) \, e^{(u'' - u') \,{\rm cap}(B_{L'})}.
\end{equation}
Using (\ref{1.29}), (\ref{1.30}), it then follows that
\begin{equation}\label{1.39}
|\Delta ' - e^{(u'' - u') \,{\rm cap}(B_{L'})} \,\Delta ''| \le 3(u'' - u) \,\big(1 + {\rm cap} (B_{L'})^2\big)\, e^{(u'' - u') \,{\rm cap}(B_{L'})}.
\end{equation}
This completes the proof of (\ref{1.8}) and hence of Lemma \ref{lem1.3}. \hfill $\square$

With this last ingredient the proof of Proposition \ref{prop1.2} is now complete. \hfill $\square$

\section{The variational problem}
The main object of this section is to prove Theorem \ref{0.2} that provides a notion of minimizers for the variational problem (\ref{0.9}), see (\ref{0.13}) - (\ref{0.15}). At the end of the section, the Remark \ref{rem2.1} contains additional information on the variational problem, in particular when $D$, see (\ref{0.7}), is star-shaped or a ball.

{\it Proof of Theorem \ref{0.2}:} We will first prove (\ref{0.14}) and (\ref{0.15}). We consider $D, u, \nu$ as in (\ref{0.7}), (\ref{0.8}) and $\ov{J}\,^{\!D}_{u,\nu}$ defined in (\ref{0.12}). We let $\varphi_n \ge 0$ in $D^1(\IR^d)$, $n \ge 0$, stand for a minimizing sequence of (\ref{0.12}). Then, by Theorem 8.6, p.~208 and Corollary 9.7, p.~212 of \cite{LiebLoss01}, we can extract a subsequence still denoted by $\varphi_n$ and find $\varphi \ge 0$ in $D^1(\IR^d)$ such that $\frac{1}{2d} \int_{\IR^d} |\nabla \varphi|^2 dz \le \liminf_{n} \frac{1}{2d} \int_{\IR^d} | \nabla \varphi_n|^2 dz = \ov{J}\,^{\!D}_{u,\nu} $ and $\varphi_n \r \varphi$ a.e.~and in $L^2_{{\rm loc}}(\IR^d)$. Then, one has
\begin{equation}\label{2.1}
\begin{array}{lll}
\dis\int_D \hspace{-2.5ex}{-}  \;\ov{\theta}_0 \big((\sqrt{u} + \varphi)^2\big) \,dz  \ge &
\dis\int_D \hspace{-2.5ex}{-} \;\limsup\limits_n \,\ov{\theta}_0 \big((\sqrt{u} + \varphi_n)^2\big)\,dz  
\\
&\hspace{-6ex} \stackrel{\rm reverse \; Fatou}{\ge} \limsup\limits_n \,\dis\int_D \hspace{-2.5ex}{-} \; \ov{\theta}_0 \big((\sqrt{u} + \varphi_n)^2\big)\,dz \ge \nu.
\end{array}
\end{equation}
This shows that $\varphi$ is a minimizer for the variational problem in (\ref{0.12}) and (\ref{0.14}) is proved. If $\varphi$ is a minimizer for (\ref{0.12}), note that $\wt{\varphi} = \varphi \wedge (\sqrt{u}_* - \sqrt{u}) \in D^1(\IR^d)$, and using Theorem 6.17, p.~152 of \cite{LiebLoss01}, $\varphi - \wt{\varphi} = (\varphi - (\sqrt{u}_* - \sqrt{u}))_+$ and $\wt{\varphi}$ are orthogonal in $D^1(\IR^d)$. In addition, one has $\ov{\theta}_0 ((\sqrt{u} + \wt{\varphi})^2) = \ov{\theta}_0((\sqrt{u} + \varphi)^2)$ so that $\wt{\varphi}$ is a minimizer for (\ref{0.12}) as well. It follows that $\varphi = \wt{\varphi}$ (otherwise $\varphi$ would not be a minimizer). With analogous arguments, one sees that the infimum defining $\ov{J}\,^{\!D}_{u,\nu}$ in (\ref{0.12}) remains the same if one omits the condition $\varphi \ge 0$ in the right member of (\ref{0.12}). Then, using smooth perturbations in $\IR^d \backslash D$ of a minimizer $\varphi$ for (\ref{0.12}), one finds that $\varphi$ is harmonic outside $D$ and tends to $0$ at infinity (see Remark 5.10 1) of \cite{Szni19} for more details). In addition, see the same reference, $|z|^{d-2} \varphi(z)$ is bounded at infinity and hence everywhere since $\varphi$ is bounded. This completes the proof of (\ref{0.15}).

We now turn to the proof of (\ref{0.13}). As already stated above Theorem \ref{0.2}, we know by direct inspection that $I^D_{u,\nu} \ge  \ov{J}\,^{\!D}_{u,\nu}$. Thus, we only need to show that
\begin{equation}\label{2.2}
\ov{J}\,^{\!D}_{u,\nu} \ge I^D_{u,\nu}.
\end{equation}
To this end, we consider a minimizer $\varphi$ for $ \ov{J}\,^{\!D}_{u,\nu}$ and know that (\ref{0.15}) holds. As we now explain, if $\psi \ge 0$ belongs to $C^\infty_0(\IR^d)$ and $\psi > 0$ on $D$, then one has
\begin{equation}\label{2.3}
\dis\int_D \hspace{-2.5ex}{-}\; \theta_0 \big((\sqrt{u} + \varphi + \psi)^2\big)\,dz > \nu.
 \end{equation}
We consider two cases to argue (\ref{2.3}). Letting $m_D$ stand for the normalized Lebesgue measure on $D$, either
\begin{eqnarray} \label{2.4}
&&m_D(\varphi < \sqrt{u}_* - \sqrt{u}) = 0 \;\;\mbox{or}
\\[1ex]
&&m_D(\varphi < \sqrt{u}_* - \sqrt{u}) > 0 .\label{2.5}
\end{eqnarray}
In the first case (\ref{2.4}), then $\varphi \ge \sqrt{u}_* - \sqrt{u}$ a.e.~on $D$ so that the left member of (\ref{2.3}) equals $1$ and (\ref{2.3}) holds since $\nu < 1$ by (\ref{0.8}). In the second case (\ref{2.5}), since $\theta_0$ is strictly increasing on $[0,u_*)$ (cf.~Lemma \ref{lem1.1}), one has
\begin{equation}\label{2.6}
\begin{array}{l}
\dis\int_D  \; \theta_0 \big((\sqrt{u} + \varphi + \psi)^2\big)\,dz  = 
\\[2ex]
\dis\int_{D \cap \{\varphi < \sqrt{u}_* - \sqrt{u}\}} \theta_0  \big((\sqrt{u} + \varphi + \psi)^2\big)\,dz +
\\[2ex]
\dis\int_{D \cap \{\varphi \ge \sqrt{u}_* - \sqrt{u}\}} \theta_0 \big((\sqrt{u} + \varphi + \psi)^2\big)\,dz >
\\[2ex]
\dis\int_{D \cap \{\varphi < \sqrt{u}_* - \sqrt{u}\}} \theta_0 \big((\sqrt{u} + \varphi )^2\big) \,dz +
|D \cap \{\varphi \ge \sqrt{u}_* - \sqrt{u}\}| =
\\[3ex]
 \dis\int_D \ov{\theta}_0 \big((\sqrt{u} + \varphi )^2\big) \,dz \ge \nu \, |D|,
\end{array}
\end{equation}
and (\ref{2.3}) follows. We have thus proved (\ref{2.3}). Using multiplication by a smooth compactly supported $[0,1]$-valued function and convolution, we can construct a sequence $\varphi_n \ge 0$ in $C^\infty_0 (\IR^d)$, which approximates $\varphi + \psi$ in $D^1(\IR^d)$ and such that $\varphi_n$ converges to $\varphi + \psi$
 a.e.~on $D$. Then, we have
\begin{equation}\label{2.7}
\begin{array}{lll}
\nu \stackrel{(\ref{2.3})}{<}  
\dis\int_D \hspace{-2.5ex}{-} \; 
\theta_0 \big((\sqrt{u} + \varphi + \psi)^2\big)\,dz & \le &
\dis\int_D \hspace{-2.5ex}{-} \; 
 \liminf\limits_n \theta_0 \big((\sqrt{u} + \varphi_n)^2\big)\,dz 
\\
&\!\!\!\! \stackrel{\rm Fatou}{\le} &\,\liminf\limits_n \,
\dis\int_D \hspace{-2.5ex}{-} \; 
 \theta_0 \big((\sqrt{u} + \varphi_n)^2\big)\,dz.
\end{array}
\end{equation}

\pagebreak\noindent
Hence, for infinitely many $n$, one has $I^D_{u,\nu} \le \frac{1}{2d} \int |\nabla \varphi_n|^2\,dz$, so that
\begin{equation}\label{2.8}
I^D_{u,\nu} \le \mbox{\f $\dis\frac{1}{2d}$} \; \dis\int_{\IR^d} |\nabla(\varphi + \psi)|^2\,dz.
\end{equation}
If we now let $\psi$ tend to $0$ in $D^1(\IR^d)$ and recall that $ \frac{1}{2d} \int_{\IR^d} |\nabla \varphi |^2 \, dz = \ov{J}\,^{\!D}_{u,\nu}$, we find (\ref{2.2}). This completes the proof of Theorem \ref{0.2}. \hfill  $\square$
%
\begin{remark}\label{rem2.1}  \rm 1) Note that for $D$ as in (\ref{0.7}) and $0 < u < u_*$, the non-decreasing map
\begin{equation}\label{2.9}
\nu \in [\theta_0(u),1) \longrightarrow I^D_{u,\nu} \stackrel{\rm Theorem\; \ref{theo0.2}}{=} \,\ov{J}\,^{\!D}_{u,\nu}  \;\mbox{is continuous}.
\end{equation}
Indeed, by definition of $I^D_{u,\nu}$ in (\ref{0.9}), the map is right continuous. To see that the map is also left continuous, consider $\nu \in (\theta_0(u),1)$ and a sequence $\nu_n$ smaller than $\nu$ increasing to $\nu$. If $\varphi_n$ is a corresponding sequence of minimizers for (\ref{0.14}), by the same arguments as above (\ref{2.1}), we can extract a subsequence still denoted by $\varphi_n$ and find $\varphi \ge 0$ in $D^1(\IR^d)$ so that $\frac{1}{2d} \int_{\IR^d} |\nabla \varphi|^2 \,dz \le \liminf_{n} \int_{\IR^d} | \nabla \varphi_n|^2\,dz = \lim_n \ov{J}\,^{\!D}_{u,\nu_n}$ and $\varphi_n \r \varphi$ a.e. Using the reverse Fatou inequality as in (\ref{2.1}), we then have
\begin{equation}\label{2.10}
\begin{array}{lll}
\dis\int_D \hspace{-2.5ex}{-} \;  \,\ov{\theta}_0 \big((\sqrt{u} + \varphi)^2\big) \,dz & \ge &\dis\dis\int_D \hspace{-2.5ex}{-} \;  \limsup\limits_n \,\ov{\theta}_0  \big((\sqrt{u} + \varphi_n)^2\big) \,dz
\\
& \ge& \limsup\limits_n \dis\int_D \hspace{-2.5ex}{-} \;  \ov{\theta}_0 \big((\sqrt{u} + \varphi_n)^2\big) \,dz \ge \limsup\limits_n \nu_n = \nu.
\end{array}
\end{equation}
This shows that $\ov{J}\,^{\!D}_{u,\nu} \le \lim_n \ov{J}\,^{\!D}_{u,\nu_n}$ and completes the proof of (\ref{2.9}).

2) If $D$ in (\ref{0.7}) is star-shaped around $z_* \in D$ (that is, when $\lambda(z-z_*) + z_* \in D$ for all $z \in D$ and $0 \le \lambda \le 1$), then for $u, \nu$ as in (\ref{0.8}), one has the additional fact
\begin{eqnarray} \label{2.9a}
&&\mbox{any minimizer $\varphi$ in (\ref{0.14}) satisfies $\dis\int_D \hspace{-2.5ex}{-} \; \ov{\theta}_0 \big((\sqrt{u} + \varphi)^2\big)\,dz = \nu$, and}
\\
&& \ov{J}\,^{\!D}_{u,\nu} = \label{2.10a}
 \\
&&  \min \Big\{ \mbox{\f $\dis\frac{1}{2d}$} \dis\int_{\IR^d} | \nabla \varphi |^2 \, dz; \; \varphi \ge 0, \varphi \in D^1 (\IR^d), \;\dis\int_D \hspace{-2.5ex}{-} \; \ov{\theta}_0 \big((\sqrt{u} + \varphi)^2\big)\,dz = \nu\Big\}. \nonumber
  \end{eqnarray}
Indeed, if $\varphi$ is a minimizer of (\ref{0.14}), one sets for $0 < \lambda < 1$, $\varphi_\lambda(z) = \varphi (z_* + \frac{1}{\lambda} \;(z - z_*)$). Then, one has $\int_{\IR^d} |\nabla \varphi_\lambda |^2 \, dz = \lambda^{d-2} \int_{\IR^d} |\nabla \varphi |^2 \, dz$, and, with $D_\lambda \supseteq D$, the image of $D$ under the dilation with center $z_*$ and ratio ${\lambda}^{-1}$, one finds $\int_D \hspace{-2.5ex}{-} \; \ov{\theta}_0  ((\sqrt{u} + \varphi_\lambda)^2)\,dz = \int_{D_\lambda}\hspace{-3.5ex}{-} \;\;\ov{\theta}_0\; ((\sqrt{u} + \varphi)^2)\,dz \ge \lambda^d \;\int_D \hspace{-2.5ex}{-} \;\ov{\theta}_0 ((\sqrt{u} + \varphi)^2)\,dz$. Thus $\int_D \hspace{-2.5ex}{-} \;\;\ov{\theta}_0 ((\sqrt{u} + \varphi)^2)\, dz \ge \nu$ must actually equal $\nu$, otherwise the consideration of $\varphi_\lambda$ for $\lambda < 1$ close to $1$ would contradict the fact that $\varphi$ is a minimizer for (\ref{0.14}). This proves (\ref{2.9a}) and (\ref{2.10a}) readily follows.

Incidentally, note that due to (\ref{2.9a}), (\ref{2.10a}),
\begin{equation}\label{2.10b}
\mbox{the map in (\ref{2.9}) is strictly increasing}.
\end{equation}
Indeed, otherwise there would be $\nu < \nu'$ with $\ov{J}\,^{\!D}_{u,\nu} =  \ov{J}\,^{\!D}_{u,\nu'}$, and corresponding minimizers $\varphi, \varphi '$ as in (\ref{2.10a}). But then $\varphi'$ would contradict (\ref{2.9a}). The claim (\ref{2.10b}) thus follows.

3) If $D$ satisfying (\ref{0.7}) is a closed Euclidean ball of positive radius in $\IR^d$, given a minimizer $\varphi$ of (\ref{0.14}), we can consider its symmetric decreasing rearrangement $\varphi^*$ relative to the center of $D$, see Chapter 3 \S3 of \cite{LiebLoss01}. One knows that $\varphi^* \in D^1(\IR^d)$ and $\int_{\IR^d} |\nabla \varphi^*|^2\, dz \le \int_{\IR^d} |\nabla \varphi |^2\,dz$, see p.~188-189 of the same reference. As we now explain:
\begin{equation}\label{2.11}
\mbox{$\varphi^*$ is a minimizer of (\ref{0.14}) as well}.
\end{equation}

\n
The argument is a (small) variation on Remark 5.10 2) of \cite{Szni19}. With $m_D$ the normalized Lebesgue measure on $D$, one has $m_D(\varphi \ge s) \le m_D(\varphi^* \ge s)$ for all $s$ in $\IR$. Setting $\ov{\theta}^{-1}_0 (a) = \inf\{ t \ge 0$; $\ov{\theta}_0(t) \ge a\}$, for $0 \le a \le 1$, we see that for $0 \le t \le 1$, $\{\ov{\theta}_0((\sqrt{u} + \varphi)^2) \ge t\} = \{ \varphi \ge \sqrt{\ov{\theta}_0^{-1}(t)} - \sqrt{u}\}$, and a similar identity holds with $\varphi^*$ in place of $\varphi$. Hence, we have
\begin{equation}\label{2.12}
\begin{array}{lll}
\nu& \le &\dis\int_D \hspace{-2.5ex}{-} \; \ov{\theta}_0 \big((\sqrt{u} + \varphi)^2\big)\,dt = \dis\int^1_0 m_D \big(\ov{\theta}_0 \big((\sqrt{u} + \varphi)^2\big) \ge t\big)\,dt  
\\[1ex]
&=& \dis\int^1_0 m_D \big(\varphi \ge \sqrt{\ov{\theta}_0^{-1}(t)} - \sqrt{u}\big)\,dt \le \dis\int^1_0 m_D \big(\varphi^* \ge \sqrt{\ov{\theta}_0^{-1}(t)} - \sqrt{u}\big)\,dt
\\[1ex]
&= & \dis\int^1_0 m_D \big(\ov{\theta}_0 \big((\sqrt{u} + \varphi)^2\big) \ge t\big) dt 
=  \dis\int_D \hspace{-2.5ex}{-} \;
\ov{\theta}_0 \big((\sqrt{u} + \varphi^*)^2\big)\,dz \,.
\end{array}
\end{equation}

Thus, $\varphi^*$ is a minimizer of (\ref{0.14}) as well, and the claim (\ref{2.11}) follows. Incidentally, note that $D$ is clearly star-shaped so that (\ref{2.9}) and (\ref{2.10b}) hold. \hfill $\square$
\end{remark}

With Theorem \ref{0.2} we have a notion of minimizers for the variational problem corresponding to (\ref{0.9}). As mentioned in the Introduction, it is a natural question whether there is a strengthening of the asymptotics (\ref{0.10}): is it the case that
\begin{equation}\label{2.13}
\lim\limits_N \;\mbox{\f $\dis\frac{1}{N^{d-2}}$} \; \log \IP[| D_N \backslash \cC^u_\infty| \ge \nu \, |D_N|] = J^D_{u,\nu} \stackrel{\rm Theorem \ref{0.2}}{=} \ov{J}\,^{\!D}_{u,\nu}  \,?
\end{equation}
Given a minimizer $\varphi$ in (\ref{0.14}), the function $(\sqrt{u} + \varphi)^2 (\frac{\cdot}{N})$ can heuristically be interpreted as describing the slowly varying local levels of the tilted interlacements that enter the derivation of the lower bound (\ref{0.10}) for (\ref{2.13}), see Section 4 of \cite{Szni19}. Hence, the special interest in analyzing whether the minimizers $\varphi$ for (\ref{0.14}) reach the value $\sqrt{u}_* - \sqrt{u}$. Indeed, if $\varphi$ remains smaller than $\sqrt{u}_* - \sqrt{u}$ the local level function $(\sqrt{u} + \varphi)^2$ remains smaller than $u_*$, and so with values in the percolative regime of the vacant set of random interlacements. On the other hand, the presence of a region where $\varphi \ge \sqrt{u}_* - \sqrt{u}$ raises the question of the possible occurrence of droplets secluded from the infinite cluster of the vacant set that would take part in the creation of an excessive fraction $\nu$ of sites of $D_N$ outside the infinite cluster of $\cV^u$ (somewhat in the spirit of the Wulff droplet in the case Bernoulli percolation or for the Ising model, see \cite{Cerf00}, \cite{Bodi99}).

\section{An application of the $C^1$-property of $\theta_0$ to the variational problem}

The main object of this section is to prove Theorem \ref{theo0.3} of the Introduction that describes a regime {\it of small excess} $\nu$ for which all minimizers of the variational problem (\ref{0.14}) remain strictly below the threshold value $\sqrt{u}_* - \sqrt{u}$. At the end of the section, the Remark \ref{rem3.4} contains some simple observations concerning the existence of minimizers reaching the threshold value $\sqrt{u}_* - \sqrt{u}$.

We consider $D$ as in (\ref{0.7}), and as in (\ref{0.16})
\begin{equation}\label{3.1}
\mbox{$u_0 \in (0,u_*)$ such that $\theta_0$ is $C^1$ on a neighborhood of $[0,u_0]$}.
\end{equation}
To prove Theorem \ref{theo0.3}, we will replace $\theta_0$ by a suitable $C^1$-function $\wt{\theta}$, which agrees with $\theta_0$ on $[0,u_0]$, see Lemma \ref{lem3.1}, and show that for $0 < u < u_0$ and $\nu \ge \theta_0(u)$ the variational problem $\wt{J}\,^{\!D}_{u,\nu}$ attached to $\wt{\theta}$, see (\ref{3.15}) and Lemma \ref{lem3.3}, has minimizers that satisfy an Euler-Lagrange equation, see (\ref{3.19}), involving a Lagrange multiplier that can be bounded from above and below in terms of $\nu - \theta_0(u)$, see (\ref{3.20}). Using such tools, we will derive properties such as stated in (\ref{0.17}) for the minimizers of $\wt{J}\,^{\!D}_{u,\nu}$ and show that they coincide with the minimizers of the original problem $\ov{J}\,^{\!D}_{u,\nu}$ in (\ref{0.14}) when $0 < u < u_0$ and $\nu$ is close to $\theta_0(u)$, see below (\ref{3.28}).

{\it Proof of Theorem \ref{theo0.3}:} 

Recall $u_0$ as in (\ref{3.1}). Our fist step is

\bigskip
\begin{lemma}\label{lem3.1}
There exist non-negative functions $\wt{\theta}$ and $\wt{\gamma}$ on $\IR_+$ such that 
\begin{eqnarray}
&&\theta_0 = \wt{\theta}- \wt{\gamma}, \label{3.2}
\\[1ex]
&&\mbox{the function $\wt{\eta}(b) = \wt{\theta}(b^2)$ is $C^1$ on $\IR$}, \label{3.3}
\\[1ex]
&&\mbox{$\wt{\eta}\, '$ is bounded and uniformly continuous on $\IR$},\label{3.4}
\\[1ex]
&&\mbox{$\wt{\eta}\, '$ is uniformly positive on each interval $[a, + \infty), a > 0$,}\label{3.5}
\\[1ex]
&&\mbox{$\wt{\gamma} = 0$ on $[0,u_0]$ and $\wt{\gamma} > 0$ on $(u_0,\infty)$}. \label{3.6}
\end{eqnarray}
\end{lemma}

\begin{proof}
By assumption there is $u_1 \in (u_0, u_*)$ such that $\theta_0$ is $C^1$ on a neighborhood of $[0,u_1]$ with a uniformly positive derivative on $[0,u_1]$ by Lemma \ref{lem1.1}. We set $u_2 = \max \{u_*, 4\}$, so that $u_0 < u_1 < u_2$. We then define $\wt{\theta}(v) = \theta_0(v)$ on $[0,u_0]$, $\wt{\theta}(v) = \theta_0(v) + a(v - u_0)^2$ on $[u_0,u_1]$, where $a > 0$ is chosen so that $\wt{\theta}(u_1) = 1$ ($\ge \theta_0 (u_*) > \theta_0(u_1)$), and $\wt{\theta}(v) = \sqrt{v} \;(\ge 2)$ on $[u_2,\infty)$. In particular, $\wt{\eta}(b) = b$ for $b \ge \sqrt{u}_2$. Then, any choice of $\wt{\theta}$ on $[u_1,u_2]$ that is $C^1$ on $[u_1,u_2]$ with right derivative $\theta'_0(u_1)$ at $u_1$, left derivative $\frac{1}{2 \sqrt{u}_2}$ at $u_2$, and uniformly positive derivative on $[u_1,u_2]$, leads to functions $\wt{\theta}, \wt{\gamma}$ that satisfy (\ref{3.2}) - (\ref{3.6}).
\end{proof}

\pagebreak
We select functions fulfilling (\ref{3.2}) - (\ref{3.6}) and from now on we view 
\begin{equation}\label{3.7}
\mbox{$\wt{\theta}$ (and hence $\wt{\gamma}$) as fixed and solely depending on $u_0$}.
\end{equation}
For the results below up until the end of the proof of Theorem \ref{theo0.3}, the only property of $u_0$ that matters is that $u_0$ is positive and a decomposition of $\theta_0$ satisfying (\ref{3.2}) - (\ref{3.6}) has been selected. In particular, if such a decomposition can be achieved in the case of $u_0 = u_*$, the results that follow until the end of the proof of Theorem \ref{theo0.3}, with the exception of the last inequality (\ref{0.17}) (part of the claim at the end of the proof), remain valid. This observation will be useful in Remark \ref{rem3.4} at the end of this section.

With $u \in (0,u_0)$, $D$ as in (\ref{0.7}), and $\wt{\eta}$ as in (\ref{3.3}), we now introduce the map:
\begin{equation}\label{3.8}
\wt{A} : \varphi \in D^1(\IR^d) \r \wt{A}(\varphi) = \dis\int_D \hspace{-2.5ex}{-}\;
\,\wt{\eta} (\sqrt{u} + \varphi) \,dz \in \IR.
\end{equation}
We collect some properties of $\wt{A}$ in the next

\begin{lemma}\label{lem3.2}
\begin{eqnarray}
&&|\wt{A}(\varphi + \psi) - \wt{A} (\varphi)| \le c(u_0) \,\|\psi \|_{L^1(m_D)}, \; \mbox{for} \; \varphi, \psi \in D^1(\IR^d)\label{3.9}
\\
&&\mbox{(recall $m_D$ stands for the normalized Lebesgue measure on $D$)}. \nonumber
\\[3ex]
&&\mbox{$\wt{A}$ is a $C^1$-map and $A'(\varphi)$, its differential at $\varphi \in D^1(\IR^d)$, is the}\label{3.10}
\\
&&\mbox{ linear form} \;\psi \in D^1(\IR^d) \r \dis\int_D \hspace{-2.5ex}{-}\; \wt{\eta} \,' (\sqrt{u} +\varphi) \,\psi \,dz = A'(\varphi) \,\psi . \nonumber
\\[3ex]
&&\mbox{For any $\varphi \ge 0$, $A'(\varphi)$ is non-degenerate}. \label{3.11}
\end{eqnarray}
\end{lemma}

\begin{proof}
The claim (\ref{3.9}) is an immediate consequence of the Lipschitz property of $\wt{\eta}$ resulting from (\ref{3.4}). We then turn to the proof of (\ref{3.10}). For $\varphi, \psi$ in $D^1(\IR^d)$, we set
\begin{equation}\label{3.12}
\begin{array}{lll}
\Gamma& =  &\; \wt{A}\,(\varphi + \psi) - \wt{A}(\varphi) - \dis\int_D \hspace{-2.5ex}{-}\; \wt{\eta}\, ' (\sqrt{u} + \varphi) \, \psi \, dz = 
\\
&&\; \dis\int^1_0 ds\dis\int_D \hspace{-2.5ex}{-}\; \big(\wt{\eta}\,' (\sqrt{u} + \varphi + s\, \psi) - \wt{\eta}\,'(\sqrt{u} + \varphi)\big)\,\psi\, dz.
\end{array}
\end{equation}
With the help of the uniform continuity and boundedness of $\wt{\eta}\, '$, see (\ref{3.4}), for any $\delta > 0$ there is a $\rho > 0$ such that for any $\varphi, \psi$ in $D^1(\IR^d)$
\begin{equation}\label{3.13}
\begin{array}{lll}
|\Gamma | & \le &\dis\int_D \hspace{-2.5ex}{-}\; (\delta + 2 \| \wt{\eta} \,'\|_\infty 1\{| \psi | \ge \rho\}) \,|\psi | \, dz
\\[2ex]
& \le &\delta \| \psi \|_{L^1(m_D)}  +  \mbox{\f $\dis\frac{2}{\rho}$} \;\|\wt{\eta} \,'\|_\infty \, \| \psi\|^2_{L^2(m_D)}.
\end{array}
\end{equation}

\medskip
Since the $D^1(\IR^d)$-norm controls the $L^2(m_D)$-norm, see Theorem 8.3, p.~202 of \cite{LiebLoss01}, we see that for any $\varphi \in D^1(\IR^d)$, $\Gamma = o\, (\|\psi\|_{D^1(\IR^d)})$, as $\psi \r 0$ in $D^1(\IR^d)$. Hence, $\wt{A}$ is differentiable with differential given in the second line of (\ref{3.10}). In addition, with $\delta > 0$ and $\rho > 0$ as above, for any $\varphi, \gamma, \psi$ in $D^1(\IR^d)$
\begin{equation}\label{3.14}
\begin{array}{l}
\Big| \dis\int_D \hspace{-2.5ex}{-}\;
(\wt{\eta} \,' \big(\sqrt{u} + \varphi + \gamma) - \wt{\eta} \,' (\sqrt{u} + \varphi)\big) \,\psi \, dz \Big| \le 
\\[2ex]
\dis\int_D \hspace{-2.5ex}{-} \;(\delta + 2 \| \wt{\eta}\, '\|_\infty \,1\{| \gamma | \ge \rho\}) \,|\psi | \,dz
\\[1ex]
\le \delta \|\psi \|_{L^1(m_D)} +  \mbox{\f $\dis\frac{2}{\rho}$} \;\| \wt{\eta}\, '\|_\infty \, \|\gamma \|_{L^2(m_D)} \,\|\psi\|_{L^2(m_D)}.
\end{array}
\end{equation}
This readily implies that $\wt{A}$ is $C^1$ and completes the proof of (\ref{3.10}). Finally, (\ref{3.11}) follows from (\ref{3.5}) and the fact that $u > 0$. This completes the proof of Lemma \ref{lem3.2}.
\end{proof}

Recall that $u \in (0,u_0)$. We now define the auxiliary variational problem
\begin{equation}\label{3.15}
\begin{array}{l}
\wt{J}\,^{\!D}_{u,\nu}  = \min \Big\{ \mbox{\f $\dis\frac{1}{2d}$} \;\dis\int_{\IR^d} |\nabla \varphi |^2 dz; \, \varphi \ge 0, \varphi \in D^1(\IR^d), \wt{A}(\varphi) \ge \nu\Big\}, 
\\[2ex]
\mbox{for $\nu \ge \wt{\theta} (u)$ ($\stackrel{(\ref{3.6})}{=} \theta_0(u)$\,)}.
\end{array}
\end{equation}
In the next lemma we collect some useful facts about this auxiliary variational problem and its minimizers. We denote by $G$ the convolution with the Green function of $\frac{1}{2d}\, \Delta$ (i.e. $\frac{d}{2 \pi^{d/2}} \,\Gamma(\frac{d}{2} - 1) \, | \cdot |^{-(d-2)}$ with $| \cdot |$ the Euclidean norm on $\IR^d$).

\begin{lemma}\label{lem3.3}
For $D$ as in (\ref{0.7}), $u \in (0,u_0)$, $\nu \ge \wt{\theta}(u)$ ($= \theta_0(u)$), one has
\begin{equation}\label{3.16}
\wt{J}\,^{\!D}_{u,\nu} = \min \Big\{ \mbox{\f $\dis\frac{1}{2d}$} \;\dis\int_{\IR^d}  |\nabla \varphi |^2 dz; \, \varphi \ge 0, \varphi \in D^1(\IR^d), \wt{A}(\varphi) =\nu\Big\}.
\end{equation}
Moreover, one can omit the condition $\varphi \ge 0$ without changing the above value, and
\begin{equation}\label{3.17}
\mbox{any minimizer of (\ref{3.15}) satisfies $\wt{A} (\varphi) = \nu$}.
\end{equation}
In addition, when $\nu = \wt{\theta}(u)$, $\wt{\varphi} = 0$ is the only minimizer of (\ref{3.15}) and when $\nu > \theta_0(u)$, for any minimizer $\wt{\varphi}$ of (\ref{3.15})
\begin{equation}\label{3.18}
\begin{array}{l}
\mbox{$\wt{\varphi} \,(\ge 0)$ is $C^{1,\alpha}$ for all $\alpha \in (0,1)$, harmonic outside $D$, with}
\\
\mbox{$\sup\limits_z |z|^{d-2} \varphi(z) < \infty$},
\end{array}
\end{equation}
and there exists a Lagrange multiplier $\wt{\lambda} > 0$ such that
\begin{eqnarray}
&& \wt{\varphi} = \wt{\lambda} \,G( \wt{\eta}\,' (\sqrt{u} + \wt{\varphi}) \,1_D), \;\mbox{with} \label{3.19}
\\[1ex]
&&c'(u_0,D) \,\big(\nu - \theta_0(u)\big) \le \wt{\lambda} \le c(u,u_0,D) \,\big(\nu - \theta_0(u)\big)\label{3.20}
\end{eqnarray}
(recall that $\theta_0(u) = \wt{\theta}(u)$). 
\end{lemma}

\begin{proof}
We begin by the proof of (\ref{3.16}), (\ref{3.17}). For $\varphi \in D^1(\IR^d)$, we write $\cD(\varphi)$ as a shorthand for $\frac{1}{2d} \int_{\IR^d} | \nabla \varphi |^2 dz$. Note that $\lim_{b \r \infty} \wt{\eta}(b) = \infty$ by (\ref{3.5}), so that the set in the right member of (\ref{3.15}) is not empty. Taking a minimizing sequence $\varphi_n$ in (\ref{3.15}), we can extract a subsequence still denoted by $\varphi_n$ and find $\varphi \in D^1(\IR^d)$ such that $\cD(\varphi) \le \liminf_n \cD(\varphi_n)$ and $\varphi_n \r \varphi$ in $L^1(m_D)$ (see Theorem 8.6, p.~208 of \cite{LiebLoss01}). By (\ref{3.9}) of Lemma \ref{lem3.2}, we find that $\wt{A}(\varphi) \ge \nu$. Hence, $\varphi$ is a minimizer of (\ref{3.15}).

Now, for any minimizer $\varphi$ of (\ref{3.15}), if $\wt{A}(\varphi) > \nu$, then for some $\lambda \in (0,1)$ close to $1$, $\wt{A} (\lambda \varphi) \ge \nu$. Moreover, $\varphi$ is not the zero function (since $\wt{A} (\varphi) > \nu)$, and $\cD( \lambda \varphi) = \lambda^2 \cD(\varphi) < \cD (\varphi)$. This yields a contradiction and (\ref{3.17}), (\ref{3.16}) follow.

Also, if one removes the condition $\varphi\ge 0$ in (\ref{3.17}), one notes that for any $\varphi$ in $D^1(\IR^d)$, $\cD(|\varphi |) \le \cD(\varphi)$ and $\wt{A}(|\varphi |) \ge \wt{A}(\varphi)$. So, the infimum obtained by removing the condition $\varphi \ge 0$ is at least $\wt{J}\,^{\!D}_{u,\nu} $ and hence equal to $\wt{J}\,^{\!D}_{u,\nu}$. The claim of Lemma \ref{lem3.3} below (\ref{3.16}) follows.

When $\nu = \wt{\theta}(u)$, $\wt{J}\,^{\!D}_{u,\nu} = 0$ and $\varphi = 0$ is the only minimizer. We now assume $\nu > \wt{\theta}(u)$ and will prove (\ref{3.18}), (\ref{3.19}). For $\wt{\varphi} \ge 0$ in $D^1(\IR^d)$ a minimizer of (\ref{3.16}), one finds using smooth perturbations in $\IR^d \backslash D$ (see Remark 5.10 1) of \cite{Szni19} for similar arguments) that $\wt{\varphi}$ is a non-negative harmonic function in $\IR^d \backslash D$ that vanishes at infinity and that $|z|^{d-2} \,\wt{\varphi}(z)$ is bounded at infinity. By (\ref{3.10}), (\ref{3.11}) of Lemma \ref{lem3.2}, $\wt{\varphi}$ satisfies an Euler-Lagrange equation (see Remark 5.10 4) of \cite{Szni19} for a similar argument) and for a suitable Lagrange multiplier $\wt{\lambda}$, one has (\ref{3.19}) (and necessarily $\wt{\lambda} > 0$). Since $\wt{\eta} \,'$ is bounded by (\ref{1.4}), it follows from (\ref{3.19}) that $\wt{\varphi}$ is $C^{1,\alpha}$ for all $\alpha \in (0,1)$, see for instance (4.8), p.~71 of \cite{GilbTrud83}. This proves (\ref{3.18}), (\ref{3.19}).

There remains to prove (\ref{3.20}). We have (recall that $\theta_0(u) = \wt{\theta}(u)$)
\begin{equation}\label{3.21}
\nu - \theta_0(u) = \dis\int_D \hspace{-2.5ex}{-}\; \wt{\eta} (\sqrt{u} + \wt{\varphi}) - \wt{\eta} \,(\sqrt{u}) \,dz .
\end{equation}
By (\ref{3.4}), we see that
\begin{equation}\label{3.22}
\begin{array}{lll}
\nu - \theta_0(u)& \le & \; \| \wt{\eta} \,'\|_\infty
\dis\int_D \hspace{-2.5ex}{-}\;
\wt{\varphi} \,dz \stackrel{(\ref{3.19})}{=} \wt{\lambda} \,\|\wt{\eta} \,'\|_\infty 
\dis\int_D \hspace{-2.5ex}{-}\;
G(\wt{\eta}\,' \big(\sqrt{u} + \wt{\varphi}) \,1_D\big) \,dz
\\
&\le &\; \wt{\lambda} \,\|\wt{\eta} \,'\|_\infty^2
\dis\int_D \hspace{-2.5ex}{-}\;  G(1_D) \,dz = c(u_0,D) \,\wt{\lambda}.
\end{array}
\end{equation}
On the other hand, by (\ref{3.5}), we see that
\begin{equation}\label{3.23}
\begin{array}{lll}
\nu - \theta_0(u) &\ge & \;\inf\limits_{[\sqrt{u},\infty)}  \wt{\eta} \,'  \dis\int_D \hspace{-2.5ex}{-}\;  \wt{\varphi} \,dz \stackrel{(\ref{3.19})}{=} \wt{\lambda} \,\inf\limits_{[\sqrt{u},\infty)} \wt{\eta} \,' \dis\int_D \hspace{-2.5ex}{-}\;  G\big(\wt{\eta} \,'  (\sqrt{u} + \varphi^2) \, 1_D\big)  \,dz  
\\
&\ge &\; \wt{\lambda} \,\big(\inf\limits_{[\sqrt{u},\infty)} \wt{\eta} \,'\big)^2\dis\int_D \hspace{-2.5ex}{-}\;  G(1_D) \,dz = c(u,u_0,D) \,\wt{\lambda}.
\end{array}
\end{equation}
The claim (\ref{3.20}) now follows from (\ref{3.22}) and (\ref{3.23}). This concludes the proof of Lemma~\ref{lem3.3}. \hfill $\square$

We now continue the proof of Theorem \ref{theo0.3}. Given $u \in (0,u_0)$ and $\nu \ge \wt{\theta}(u)$ ($= \theta_0(u)$), we see by Lemma \ref{lem3.3} that any minimizer $\wt{\varphi}$ for (\ref{3.16}) satisfies (\ref{3.19}) for a suitable $\wt{\lambda}$ satisfying (\ref{3.20}), so that
\begin{equation}\label{3.24}
\|\wt{\varphi}\|_\infty \le \wt{\lambda} \,\|  \wt{\eta} \,'\|_\infty \; \|G1_D\|_\infty \stackrel{(\ref{3.20}), (\ref{3.4})}{\le} c_2(u,u_0,D) \big(\nu - \theta_0(u)\big).
\end{equation}
In particular, we find that
\begin{equation}\label{3.25}
\begin{array}{l}
\mbox{for $\theta_0(u) \le \nu \le \theta_0(u) + c_1(u,u_0,D) (< 1)$, any minimizer $\wt{\varphi}$}
\\
\mbox{for (\ref{3.16}) satisfies $0 \le\wt{\varphi} \le (\sqrt{u}_0 - \sqrt{u}) \wedge \big\{c_2\big(\nu - \theta_0(u)\big)\big\}$}.
\end{array}
\end{equation}
We will now derive the consequences for the basic variational problem of interest $\ov{J}\,^{\!D}_{u,\nu}$, see (\ref{0.12}), (\ref{0.14}). By (\ref{3.2}), (\ref{3.6}) and the definition of $\ov{\theta}_0$ (see (\ref{0.11})), we find that $\wt{\theta} \ge \ov{\theta}_0$, so that
\begin{equation}\label{3.26}
\mbox{for all $u \in (0,u_0)$ and $\nu \in [\theta_0(u), 1),\, \ov{J}\,^{\!D}_{u,\nu} \ge \wt{J}\,^{\!D}_{u,\nu}$}.
\end{equation}
Moreover, when $\nu \in [\theta_0(u), \theta_0(u) + c_1]$ (with $c_1$ as in (\ref{3.25})), any minimizer $\wt{\varphi}$ for (\ref{3.16}) is bounded by $\sqrt{u}_0 - \sqrt{u}$, and hence satisfies as well $\int_D \hspace{-2.5ex}{-}\;
\ov{\theta}_0 ((\sqrt{u} + \wt{\varphi})^2) \,dz \ge \nu$ (in fact an equality by (\ref{3.17})). We thus find that
\begin{equation}\label{3.27}
\begin{array}{l}
\mbox{$\ov{J}\,^{\!D}_{u,\nu} = \wt{J}\,^{\!D}_{u,\nu}$ for all $\nu \in [\theta_0(u) + c_1]$, and any minimizer $\wt{\varphi}$ of $\wt{J}\,^{\!D}_{u,\nu}$ in (\ref{3.16})}
\\
\mbox{is a minimizer of $\ov{J}\,^{\!D}_{u,\nu}$ in (\ref{0.14})}.
\end{array}
\end{equation}
Now for $\nu$ as above, consider $\varphi$ a minimizer of (\ref{0.14}). Then, we have $\cD(\varphi) = \ov{J}\,^{\!D}_{u,\nu} = \wt{J}\,^{\!D}_{u,\nu}$, and since $\wt{\theta} \ge \ov{\theta}_0$, we find that
\begin{equation}\label{3.28}
\wt{A}(\varphi) = \dis\int_D \hspace{-2.5ex}{-}\;
 \wt{\theta}\,\big((\sqrt{u} + \varphi)^2\big)\,dz \ge \dis\int_D \hspace{-2.5ex}{-}\;
  \ov{\theta}_0 \big((\sqrt{u} + \varphi)^2\big)\,dz \ge \nu.
\end{equation}

\n
This show that $\varphi$ is a minimizer for (\ref{3.15}), hence for (\ref{3.16}) by (\ref{3.17}). We thus find that when $\nu \in [\theta_0(u), \theta_0(u) + c_1]$, the set of minimizers of (\ref{0.14}) and (\ref{3.16}) coincide and the claim (\ref{0.17}) now follows from Lemma \ref{lem3.3}. This concludes the proof of Theorem \ref{theo0.3}.
\end{proof}

With Theorem \ref{theo0.3} we have singled out a regime of ``small excess'' for $\nu$ such that all minimizers $\varphi$ for $\ov{J}\,^{\!D}_{u,\nu}$ in (\ref{0.14}) stay below the maximal value $\sqrt{u}_* - \sqrt{u}$. In the remark below we make some simple observations about the possible existence of a regime where some minimizers in (\ref{0.14}) reach the threshold $\sqrt{u}_* - \sqrt{u}$.

\begin{remark}\label{rem3.4} \rm 1) If $\theta_0$ is discontinuous at $u_*$ (a not very plausible possibility), then $\theta_0 (u_*) < 1$, and for any $\nu \in (\theta_0(u_*),1)$ any minimizer for (\ref{0.14}) must reach the threshold value $\sqrt{u}_* - \sqrt{u}$ on a set of positive Lebesgue measure due to the constraint in (\ref{0.14}).

2) If $\theta_0$ is continuous and its restriction to $[0,u_*]$ is $C^1$ with uniformly positive derivative (corresponding to a ``mean field'' behavior of the percolation function $\theta_0$), then a decomposition as in Lemma \ref{lem3.1} can be achieved with now $u_0 = u_*$. As mentioned below (\ref{3.7}), the facts established till the end of Theorem \ref{theo0.3} (with the exception of the last inequality of (\ref{0.17})) remain valid in this context. In particular, if for some $u \in (0,u_*)$ and $\nu \in (\theta_0(u), 1)$ there is a minimizer $\wt{\varphi}$ for $\wt{J}\,^{\!D}_{u,\nu}$ in (\ref{3.16}) such that $\|\wt{\varphi}\|_\infty = \sqrt{u}_* - \sqrt{u}$, then $\wt{\varphi}$ is a minimizer for $\ov{J}\,^{\!D}_{u,\nu}$ in (\ref{0.14}) and it reaches the threshold value $\sqrt{u}_* - \sqrt{u}$. In the toy example where $\wt{\eta}$ is affine on $[\sqrt{u} + \infty)$ and $0 < \wt{\eta}\,(\sqrt{u}) < \wt{\eta}\, (\sqrt{u}_*) = 1$, such a $\nu < 1$ and $\wt{\varphi}$ (which satisfies (\ref{3.19})) are for instance easily produced. \hfill $\square$
\end{remark}

The above remark naturally raises the question of finding some plausible assumptions on the behavior of the percolation function $\theta_0$ close to $u_*$ (if the behavior mentioned in Remark \ref{rem3.4} 2) is not pertinent, see for instance Figure 4 of \cite{MariLebo06} for the level-set percolation of the Gaussian free field, when $d=3$) and whether such assumptions give rise to a regime for $u,\nu$, ensuring that minimizers of $\ov{J}\,^{\!D}_{u,\nu}$ in (\ref{0.14}) achieve the maximal value $\sqrt{u}_* - \sqrt{u}$ on a set of positive measure. But there are many other open questions. For instance, what can be said about the number of minimizers for (\ref{0.14})? Is the map $\nu \r \ov{J}\,^{\!D}_{u,\nu}$ in (\ref{2.9}) convex?  An important question is of course whether the asymptotic lower bound (\ref{0.10}) can be complemented by a matching asymptotic upper bound.

\end{document}